\documentclass[11pt]{amsart}

\usepackage[utf8]{inputenc}
\usepackage[english]{babel}
\usepackage{amssymb}
\usepackage{mathtools}
\usepackage{mathrsfs}
\usepackage[mathscr]{euscript}
\usepackage{stmaryrd}
\usepackage{enumitem}
\usepackage{enumerate}
\usepackage{lmodern}
\usepackage[all]{xy}
\usepackage{rotating}
\usepackage{extarrows}
\usepackage{leftidx}
\usepackage{upgreek}
\usepackage{bbm}
\usepackage{bbold}
\usepackage{xr-hyper}
\usepackage{tikz}
\usetikzlibrary{matrix,arrows}
\usepackage{tikz-cd}
\usepackage{hyperref}
\usepackage{macros}
\newcommand{\comma}{,}

\hypersetup{
 pdftitle    = {Geometric Brauer residue via root stacks},
  pdfauthor   = {Jakob Oesinghaus},
  colorlinks=true}
\author{Jakob Oesinghaus}
\address{
  Jakob Oesinghaus,
  Institut f\"ur Mathematik,
  Universit\"at Z\"urich,
  Winterthurerstrasse 190,
  8057 Z\"urich,
  Switzerland,
  jakob.oesinghaus@math.uzh.ch
}

\title{Geometric Brauer residue via root stacks}

\begin{document}
\begin{abstract}
We reinterpret the residue map for the Brauer group of a smooth variety using a root stack construction
and Weil restriction for algebraic stacks, and apply the result to find a geometric representative of the
residue of the Brauer class associated to a conic bundle.
\end{abstract}
\maketitle

\tableofcontents

\section{Introduction}
Consider a variety $S$ over a field $K$, and let $n$ be a positive integer not
divisible by the characteristic of $K$. The \emph{Brauer group} $\Br(S)$ is a classical invariant
studied, among other things, in the context of rationality questions, since it is a 
birational invariant for smooth projective varieties. More precisely, it can be shown that
for a \emph{smooth projective} variety $X$ over an algebraically closed field,
the $n$-torsion part of the Brauer group $\Br(X)[n]$ agrees with the group of
\emph{unramified} $n$-torsion elements of the Brauer group of the function field, defined as
\[
\bigcap_{R \subset K(X)} \Ker \big(\Br(K(X))[n] \to H^1(\kappa_R, \ZnZ) \big),
\]
where the intersection runs over all (rank one) discrete valuation rings
$R\subset K(X)$ such that $K\subset R$, and the map
$\Br(K(X))[n] \to H^1(\kappa_R, \ZnZ)$ is the residue map for this discrete
valuation ring.

Let $S$ be a smooth variety over $K$. 
Then, for an irreducible divisor $S$, the
residue map (\cite[Prop.\ 2.1]{GB3})
\[
\res \colon \Br(K(S)) [n] \to H^1(K(D), \ZnZ), 
\]
measures the ramification of this Brauer group element along the divisor.
For example, if $\alpha \in \Br(K(S))[n]$ is the class of a Brauer-Severi scheme
of relative dimension $n-1$ over an open subset of $S$ which arises as the restriction
of a flat bundle on all of $S$, then the residue map will yield some information about the
degeneration of this bundle along the boundary (cf \cite{ArtinLeftIdeals}).
In general, the map $\res$ is hard to compute explicitly. Root stacks are well-adapted to residue calculations, since
in the setting of a discrete valuation ring, every $n$-torsion Brauer class over the generic point extends uniquely
to the $n$-fold root stack along the closed point.

We use the root stack construction over a discrete valuation ring to reinterpret the residue map
in terms of a canonical decomposition of the Brauer group of this root stack
(Theorem \ref{thm.residue-rootstack}), and use results on Weil restriction for the
gerbe of the root stack, which is just $B\mu_n$ over the residue field of the DVR, to give a
representative, in terms of a $\ZnZ$-torsor, for the residue class in this situation
(Proposition \ref{prop.weilres}).

We then apply this result in the geometric situation of a Brauer class associated to
a bundle whose generic fiber is a form of projective space.
We show that via the operations of Weil restriction and coarse moduli space of
an algebraic stack, the residue of the Brauer class along a divisor arises geometrically
from the $\bP^{n-1}$-bundle associated with the restriction of the Brauer class to the gerbe of the root stack
in Proposition \ref{prop.brauerseveri}.

As a special case, we recover a classical result by Artin(\cite{ArtinLeftIdeals}) using different methods.

\subsubsection*{Acknowledgements}
I would like to thank my advisor Andrew Kresch for the
inspiration for this note, and for me assisting me tirelessly with the technical details. I also want to thank
the anonymous referee for very helpful feedback that helped me improve the exposition.
I am supported by Swiss National Science Foundation grant 156010.

\vspace{1em}

\section{The residue map and root stacks}
Unless mentioned otherwise, all cohomology groups are \'etale cohomology groups.
We define the (cohomological) \emph{Brauer group} $\Br(X)$
of a Noetherian Deligne-Mumford stack $X$ (\cite{DeligneMumford}, \cite[IV]{MilneEtale}) to be
the torsion group $H^2(X, \bG_m)_\tors$.

Let $R$ be a DVR with fraction field $K$ and residue field $\kappa$, and let $n>0$ be an integer
not divisible by $\characteristic(\kappa)$. In this situation, by \cite[Prop.\ 2.1]{GB3}, we have
the $n$-torsion residue map
\begin{equation}
\label{eqn.residue}
\res \colon \Br(K) [n] \to H^1(\Spec(\kappa), \ZnZ),
\end{equation}
which is part of an exact sequence
\[
0 \to \Br(R)[n] \to \Br(K)[n] \to H^1(\Spec(\kappa), \ZnZ) \to 0.
\]
We recall the construction of the map $\res$. Using the Leray spectral sequence for
the inclusion $i: \Spec(K) \to \Spec(R)$ and the fact that $R^\ell i_* \bG_m [n] = 0$
for $\ell > 0$,
one deduces an exact sequence
\[
0 \to \Br (R)[n] \to \Br (K)[n] \to H^2(\Spec(R), j_* \bZ)[n] \to 0,
\]
where $j$ is the inclusion of the closed point $\Spec(\kappa)$.
Finally, a short exact sequence argument shows that
\[
H^2(\Spec(R), j_* \bZ)[n] = H^2(\Spec(\kappa), \bZ)[n] = H^1(\Spec(\kappa), \ZnZ),
\]
which establishes the residue map.

In the same situation as before, we now construct another morphism
\[
\Br(K) [n] \to H^1(\Spec(\kappa), \ZnZ)
\]
via root stack methods(\cite{GromovWittenDM,CadmanTangency}).
Concretely, the $n$-th root stack of $\Spec(R)$ along $\Spec(\kappa)$ is the Deligne-Mumford stack
\[
X = \sqrt[n]{(\Spec(R), \Spec(\kappa))} := [\Spec(R[T]/(T^n-\pi))/\mu_n] \to \Spec(R),
\]
which has the property that all $n$-torsion elements of the Brauer group of $K$ lift to it
\cite[\S 3.2]{PeriodIndexBrauerArithmeticSurface}:
\begin{equation}
\label{eqn.root-iso}
\Br(K) [n] = \Br(X)[n].
\end{equation}

The root stack is an isomorphism over $\Spec (K)$,
and the complement of $\Spec (K)$ in the root stack is the \emph{gerbe of the root stack},
a closed substack of $X$ mapping to $\Spec(\kappa)$, which is isomorphic to the 
classifying stack 
\[
q: B\mu_{n,\kappa} \to \Spec(\kappa).
\]
We can apply the Leray spectral sequence for $q$
and the sheaf $\bG_m$, the fact that $q$ has a section, and the fact that the $E_2^{0,2}$-term
of this spectral sequence vanishes, to produce a decomposition of
$H^2(B\mu_{n,\kappa}, \bG_m) = \Br(B\mu_{n,\kappa})$
as
\begin{equation}
\label{eqn.gerbe-brauer}
\Br(B\mu_{n,\kappa}) \cong \Br(\kappa) \oplus H^1(\Spec(\kappa),  \ZnZ),
\end{equation}
where the projection to the first summand arises from the section of $q$, and the projection to the second summand
is the morphism
\[
H^2(B\mu_{n,\kappa}, \bG_m) \to H^1(\Spec(\kappa), R^1q_*\bG_m) = H^1(\Spec(\kappa), \ZnZ)
\]
arising from the Leray spectral sequence.

\section{Main results}
Our main theorem shows how the two constructions from the previous section are related.

\begin{theorem}
\label{thm.residue-rootstack}
The residue map \eqref{eqn.residue} agrees with the composition of the isomorphism \eqref{eqn.root-iso},
restriction to $B\mu_{n,\kappa}$, isomorphism \eqref{eqn.gerbe-brauer}, projection to the second factor in the
direct sum decomposition, and multiplication by $(-1)$.
Concretely, let $\alpha \in \Br(K)[n]$, and extend it to an element $\bar{\alpha} \in \Br(X)[n]$. Then
\begin{equation}
\label{eqn.residue-projection}
\res(\alpha) = - \pr_2(\bar{\alpha}\vert_{B\mu_{n,\kappa}}).
\end{equation}
\end{theorem}

The proof of Theorem \ref{thm.residue-rootstack} depends on a technical result which
forms the basis for our calculation of the right-hand side of \eqref{eqn.residue-projection}.
We first need to set up some notation. Fix a base field $\kappa$.
Applying the Leray spectral sequence for the structure morphism of the classifying stacks $B\mu_n$
and $B\ZnZ$ over $\Spec(\kappa)$, the cohomology groups
$H^1(B\mu_n,\mu_n)$ and $H^1(B\ZnZ,\ZnZ)$ decompose into
an arithmetic component (pullback to $\Spec(\kappa)$) and a geometric component
(base change to $\kappa^{\mathrm{sep}}$, where $H^1$ is identified with
group homomorphisms, cyclic of order $n$,
generated by $1_{\mu_n}$, respectively, $1_{\ZnZ}$).
We denote by $1_{\mu_n} \boxtimes 1_{\ZnZ}$ the cup product
\[
pr_1^* 1_{\mu_n} \cup pr_2^* 1_{\ZnZ} \in H^2(B(\mu_n \times \ZnZ), \mu_n).
\]

\begin{lemma}
\label{lemma.group-cohom}
Let $\kappa$ be a field, and let $\kappa^{\mathrm{sep}}$ be a separable closure of $\kappa$. The element
\begin{equation}
\label{eqn.claim}
1_{\mu_n} \boxtimes 1_{\ZnZ} \in
\ker\big(H^2(B(\mu_n \times \ZnZ), \mu_n)\to H^2(B\mu_{n,\kappa^{\mathrm{sep}}},\mu_n)\big)
\end{equation}
is mapped under the Leray spectral sequence of $p: B(\mu_n \times \ZnZ) \to B\ZnZ$
to $1_{\ZnZ} \in H^1(B\ZnZ, \ZnZ) \cong H^1(B\ZnZ, \Hom(\mu_n, \mu_n))$.
\end{lemma}
\begin{proof}
Since the class \eqref{eqn.claim} vanishes upon pullback to
$B\mu_n$, its image in
\[
H^1(B\ZnZ,\ZnZ)
\]
has vanishing arithmetic component.
So we may suppose that $\kappa$ is separably closed.
Then the Leray spectral sequence reduces to the
Lyndon-Hochschild-Serre spectral sequence.
For the computation we follow the organizational scheme of
\cite{PreuEffective} for the (standard) choices of acyclic resolutions.
To obtain $Rp_*\mu_n$ we may exploit the fact that $p$ is obtained from
$q: B\mu_n \to \Spec(\kappa)$
by \'etale base change and push forward the acyclic resolution of
$\mu_n$ on $B\mu_n$ consisting of
$\mu_n$-valued functions on $(\mu_n)^{i+1}$ in
degree $i$ for all $i\ge 0$,
with homogeneous cochains as $\mu_n$-invariants.
By writing these in their inhomogeneous form,
$Rq_*\mu_n$, and hence by pullback as well $Rp_*\mu_n$, may be expressed as
\begin{equation}
\label{Rp}
\mu_n\stackrel{0}\longrightarrow
\bigoplus_{\mu_n}\mu_n
\stackrel{(c_\beta)\mapsto \big(\frac{c_\beta c_{\beta'}}{c_{\beta\beta'}}\big)}{\relbar\joinrel\relbar\joinrel\relbar\joinrel\relbar\joinrel\longrightarrow}
\bigoplus_{\mu_n^2} \mu_n\longrightarrow\dots.
\end{equation}

\begin{figure}
\begin{tikzcd}[row sep=large, column sep=large]
\mu_n\ar[r, "0"] &
\mathrm{Map}(\ZnZ,\ZnZ) \ar[r, "-d"] \ar[ddd, dashrightarrow, bend left = 30] &
\mathrm{Map}((\ZnZ)^2,\ZnZ) \ar[ddd, dashrightarrow, bend left = 40] \\
\mu_n \ar[r, "0"] \ar[u, equal] \ar[d, equal]  &
\ZnZ \ar[r]  \ar[u, "\mathrm{constant\,maps}"] \ar[d, "\ZnZ \cong \Hom (\mu_n \comma \mu_n )"'] &
0 \ar[d] \ar[u] \\
\mu_n\ar[r, "0"] \ar[d, "\mathrm{constant}"'] &
\bigoplus_{\mu_n}\mu_n \ar[r] \ar[d, "(c_\beta)\mapsto (\mathrm{constant}\,c_{\beta})_{(\beta,b)}"'] &
\bigoplus_{(\mu_n)^2}\mu_n \ar[d, "(c_{\beta,\beta'})\mapsto (\mathrm{constant}\,c_{\beta,\beta'})_{((\beta,b),(\beta',b'))}"'] \\
\mathrm{Map}(\ZnZ,\mu_n) \ar[r] &
\bigoplus_{\mu_n\times \ZnZ}\mathrm{Map}(\ZnZ,\mu_n) \ar[r] &
\bigoplus_{(\mu_n\times \ZnZ)^2}\mathrm{Map}(\ZnZ,\mu_n)
\end{tikzcd}
\caption{The morphisms of complexes giving rise to the (dotted) morphism in the derived category.}
\label{fig1}
\end{figure}

For the map to $H^1(B\ZnZ,\ZnZ)$ mentioned in the claim, we apply a cutoff functor
(cf. \cite[1.4.8]{DeligneTH2}) to obtain the subcomplex
\[ \mu_n\stackrel{0}\longrightarrow \Hom(\mu_n,\mu_n). \]
The inclusion is represented in
Figure \ref{fig1} using an analogous
acyclic resolution of $\Hom(\mu_n,\mu_n) = \ZnZ$ to that used above,
shifted by one (leading to $-d$ in the diagram),
and a quasi-isomorphic complex to \eqref{Rp}.
With this, we may compute the morphism
\[
H^2(B\ZnZ,[\mu_n\stackrel{0}\to \ZnZ])\to H^2(B\ZnZ,Rp_*\mu_n)
\cong H^2(B(\mu_n\times \ZnZ),\mu_n).
\]
By writing down a compatible morphism from the top to the bottom complex
in Figure \ref{fig1} we compute the
image of
\[ 1_{\ZnZ}\in H^1(B\ZnZ,\ZnZ)\subset H^2(B\ZnZ,[\mu_n\stackrel{0}\to \ZnZ]), \]
represented by $(b,b')\mapsto b'-b$ in the group in the top right in
Figure \ref{fig1},
to be $1_{\mu_n}\boxtimes 1_{\ZnZ}$.
\qed
\end{proof}

\begin{proof}[Theorem \ref{thm.residue-rootstack}]
We can assume without loss of generality that $R$ is Henselian. Indeed, if $R\to R^h$ is
the Henselization of $R$, the residue fields of $R$ and $R^h$ are equal, and since the
Leray spectral sequence is functorial, the natural diagrams for the residue maps commute.
Similarly, due to the nature of the root stack construction, the map on the right-hand side
of \eqref{eqn.residue-projection} is functorial for the Henselization.

If R is Henselian, by \cite[Rem.\ III.3.11]{MilneEtale}, we have
\[
H^1(\Spec(R), \ZnZ) = H^1(\Spec(\kappa), \ZnZ)
\qquad\text{and}\qquad \Br(R) = \Br(\kappa).
\]
We will keep using these isomorphisms implicitly.
For elements of $\Br(R)[n] \subset \Br(K)[n]$, both sides of \eqref{eqn.residue-projection} are zero: $\Br(R)[n]$
is the kernel of the residue map, and elements of $\Br(R)[n] = \Br(\kappa)[n]$ end up in the first summand
of the decomposition \eqref{eqn.root-iso}. Hence, it suffices to verify the equality \eqref{eqn.residue-projection}
for a subset of elements of $\Br(K)[n]$ whose residues attain all elements of $H^1(\Spec(\kappa), \ZnZ)$.

Pick a uniformizer $\pi\in R$. The morphism $\mu_n \otimes \ZnZ \to \mu_n$ induces a cup product pairing
\[
\cup\colon H^1(\Spec(K), \mu_n) \otimes H^1(\Spec(K), \ZnZ) \to H^2(\Spec(K), \mu_n) \to \Br(K)[n].
\]
As our set of elements on which we will verify \eqref{eqn.residue-projection}, we choose those of the form
$\theta \cup \gamma$,
where $\theta$ is the class of $K(\pi^{1/n})/K$ and $\gamma$ is a class in
$H^1(\Spec(R), \ZnZ)$, with the same symbol used to denote
its restriction to $K$.

Recall that root stack $X$ is the quotient
\[
X=\sqrt[n]{(\Spec(R),\Spec(\kappa))} = [\Spec(R[T]/(T^n-\pi))/\mu_n],
\]
where $\mu_n$ acts by scalar multiplication on $T$.

Fix $\gamma\in H^1(\Spec(R),\ZnZ)$, which corresponds to
a cyclic degree $n$ \'etale $R$-algebra $S$. Then
\[
\label{eq.S-torsor}
\Spec(S[T]/(T^n-\pi)) \to \Spec(R[T]/(T^n-\pi)) \to X
\]
is a $(\mu_n \times \ZnZ)$-torsor over $X$. For this torsor, the extension
$\bar{\alpha}\in \Br(X)[n]$
is the image under $H^2(X,\mu_n)\to \Br(X)[n]$ of the
cup product class $\theta \cup \gamma$ to $X$.

The extensions of $\theta$, $\gamma$, and $\theta \cup \gamma$ to the root stack 
are pulled back from the classifying space $B(\mu_n \times \ZnZ)$
via the morphism
\begin{equation}
\label{torsorX}
X \to B(\mu_n \times \ZnZ).
\end{equation}
corresponding to this $(\mu_n \times \ZnZ)$-torsor.

In the following, we use notation taken from \cite[Exa.\ III.2.6]{MilneEtale} to
write down a representation of $\bar{\alpha}\in \Br(X)[n]$ as a
\v{C}ech 2-cocycle, by transforming the simplicial \v{C}ech associated to
\eqref{eq.S-torsor}.
We make the identification of $(\mu_n\times \ZnZ)$-torsors
\[
\Spec(S[T]/(T^n-\pi))\times_X\Spec(S[T]/(T^n-\pi)) 
= \Spec(S[T]/(T^n-\pi))\times (\mu_n\times \ZnZ),
\]
denoting an element of the $\mu_n$-factor by $\beta$ and
an element of the $\ZnZ$-factor by $b$; analogously, we use
pairs of such elements for the triple fiber product over $X$.

A representative for $\bar{\alpha}$ in this notation is
\[
(\beta^{b'})_{((\beta, b), (\beta', b'))}.
\]
We define a $\bG_m$-valued \v{C}ech 2-cocyle $\varepsilon_{b,b'}$ depending (only) on $b,b'\in {0, \dots, n-1}$:
\[
\varepsilon_{b,b'}:=
\begin{cases}
1 & \text{if $b+b'<n$},\\
\pi^{-1} & \text{if $b+b'\ge n$}.
\end{cases}
\]
The difference between these cocyles is a coboundary; indeed, the coboundary of
the 1-cochain $(\pi^{b/n})_{(\beta,b)}$ is
$(\varepsilon_{b,b'}^{-1}\beta^{b'})_{((\beta, b), (\beta', b'))}$.
Since the representative $\varepsilon$ is independent of $\beta$,
it is pulled back from the \'etale
cover $\Spec(K\otimes_R S) \to \Spec(K)$.
We can explicitly compute the residue of the class represented by the
\v{C}ech 2-cocyle which $\varepsilon$ is pulled back from,
and we find it to be $-\gamma|_{\Spec(\kappa)}$.

It remains to be shown that the right-hand side of \eqref{eqn.residue-projection} yields the same element.
To show this, we compute on $B(\mu_n \times \bZ /n \bZ)$ using
Lemma \ref{lemma.group-cohom} and pull back the result via the map \eqref{torsorX}.
\qed
\end{proof}

For the next result, we will need the notion of restriction of scalars $f_*$ for a proper
flat finitely presented morphism $f$ of algebraic stacks with finite diagonal,
as described in \cite{Hall-Rydh}.
\begin{proposition}
\label{prop.weilres}
Let $\cG\to B\mu_{n,\kappa}$ be the gerbe banded by $\mu_n$
whose class $\tilde{\alpha} \in H^2(B\mu_{n,\kappa},\mu_n)$ is the unique lift
of a Brauer class $\alpha\in \Br(B\mu_{n,\kappa})[n]$ such that
the pullback of $\tilde{\alpha}$ to $B\mu_{n,\bar{\kappa}}$ vanishes.
With notation $q\colon B\mu_{n,\kappa}\to \Spec(\kappa)$ for the structure morphism
and $[\,]$ for coarse moduli space,
$[q_*\cG]$ is a $\ZnZ$-torsor whose class is $-\mathrm{pr}_2(\alpha)$.
\end{proposition}

\begin{proof}
Adding an element of $\Br(\kappa)[n]$ to $\alpha$ does not change $[q_*\mathcal{G}]$.
So it suffices to treat the case that $\alpha$ restricts to $0$ in $\Br(\kappa)$.
Then we reduce as before to the computation in the universal case, that is, over
$B(\mu_n \times \ZnZ)$, and again reduce to carrying out the
computation when $\kappa$ is separably closed.
Let $\Gamma$ be the subgroup of $GL_n(\kappa)$ generated by
the scalar $n$th roots of unity, the permutation matrix for the
$n$-cycle $(1,2,\dots,n)$ and a diagonal matrix whose
entries are successive powers of a primitive $n$th root of unity.
Then $\Gamma$ is a central $\mu_n$-extension of $\mu_n\times \ZnZ$,
where if we take $\Gamma\to \mu_n$ defined for $A\in \Gamma$
by the ratio of successive nonzero entries of $A$, and $\Gamma\to \ZnZ$, by
the position of the nonzero entry in the first row of $A$, the class in
$H^2(\mu_n\times \ZnZ,\mu_n)$ of the group extension
\[
1\to \mu_n\to \Gamma\to \mu_n\times \ZnZ\to 1
\]
is that of the $2$-cocycle
\[ (\beta'^b)_{((\beta,b),(\beta',b'))}, \]
i.e., is $-1_{\mu_n}\boxtimes 1_{\ZnZ}\in H^2(\mu_n\times \ZnZ,\mu_n)$.

With $p\colon B(\mu_n\times \ZnZ)\to B\ZnZ$ as above, the
relative moduli space \cite[\S 3]{TwistedStable} of $p_*B\Gamma$ is the
universal $\ZnZ$-torsor $\Spec(\kappa)\to \ZnZ$.
Comparing with the computation above, we obtain the result.\qed
\end{proof}
We make a definition analogous to Brauer-Severi schemes for the case that the base is an algebraic stack.
\begin{definition}
A \emph{Brauer-Severi stack of relative dimension $n-1$} over an algebraic stack $S$ is a
smooth, proper, representable morphism $p:P\to S$ such that all geometric fibers of $p$
are projective spaces $\bP^{n-1}$.
\end{definition}
A Brauer-Severi stack over a scheme $S$ is clearly just a Brauer-Severi scheme over $S$.
The following is an adaptation of a definition that was used in \cite{KrTschBrauerSeveri} in the case
$n=3$.
\begin{definition}
Let $\kappa$ be a field, and let $n$ be a positive integer not divisible by $\characteristic(\kappa)$. Let $P$
be a Brauer-Severi variety of dimension $n-1$ over $\kappa$. We say that an action of $\mu_n$ on
$P$ is \emph{balanced} if, after passing to $\bar{\kappa}$ and
identifying $P_{\bar{\kappa}} \simeq \bP^{n-1}_{\bar{\kappa}}$, the action of $\beta \in \mu_n$
is given by
\[
(x_1:x_2:\dots:x_n)\mapsto (\beta x_1:\beta^2 x_2:\dots:x_n).
\]
We make an analogous definition for a Brauer-Severi stack of relative dimension
$n-1$ over $B\mu_{n,\kappa}$.
\end{definition}

Recall that we denote the structure morphism $B\mu_{n,\kappa} \to \Spec(\kappa)$ by $q$.
\begin{proposition}
\label{prop.balanced}
Suppose that $\alpha$ is the Brauer class of a $PGL_n$-torsor over
$B\mu_{n,\kappa}$
whose associated Brauer-Severi stack $P\to B\mu_{n,\kappa}$ is
balanced.
We endow the coarse moduli space of the Weil restriction $[q_*P]$ the structure of a
$\ZnZ$-torsor, given by translation on characters of eigenspaces for local choices of rank $n$
vector bundles $E$ with $P\cong \bP(E)$. Then the associated class in
$H^1(\Spec(\kappa),\ZnZ)$ is inverse to that of $[q_*\mathcal{G}]$,
where $\mathcal{G}$ is as in Proposition \ref{prop.weilres}.
\end{proposition}
\begin{proof}
To compute this class, we make a base change to $\cG$.
On $\cG$, there is a global choice of rank $n$ vector bundle $E$ with $P \cong \bP (E)$.
Any character of $\mu_n$ induces a section of $[q_*\bP(E)] \to [q_*\cG]$, given by eigenspaces of
this character. If we endow $[q_*P]$ with the inverse of the structure of a $\ZnZ$-torsor given by
translation on characters of eigenspaces, the composite
\[
[q_*\cG] \to [q_*P]
\]
is an equivariant isomorphism.\qed
\end{proof}

\section{Application to Brauer-Severi schemes}

We can now apply the results of the previous section to the geometric situation.
Given a Brauer class as in Propositon \ref{prop.balanced}, we can combine
Propositions \ref{prop.weilres} and \ref{prop.balanced} to allow us to
identify the class of the $\ZnZ$-torsor $[q_*P]$ with $\mathrm{pr}_2(\alpha)$.
This leads to the following result.

\begin{proposition}
\label{prop.brauerseveri}
Let $R$ be a DVR with fraction field $K$ and residue field $\kappa$,
let $n>0$ be an integer not divisible by $\characteristic(\kappa)$,
and let $\alpha\in \Br(K)[n]$, with extension to
$\bar{\alpha}\in \Br(X)[n]$, where
$X = \sqrt[n]{(\Spec(R), \Spec(\kappa))} \to \Spec(R)$.
If the restriction of $\bar{\alpha}$ to the gerbe of the root stack
$B\mu_{n,\kappa}$ is the Brauer class of a $PGL_n$-torsor
whose associated Brauer-Severi stack $P\to B\mu_{n,\kappa}$ is
balanced, then the $\ZnZ$-torsor $[q_*P]$, where
$q$ denotes the structure morphism $B\mu_{n,\kappa}\to \Spec(\kappa)$,
has the class $\mathrm{pr}_2(\alpha)\in H^1(\Spec(\kappa),\ZnZ)$,
in the notation of Theorem \ref{thm.residue-rootstack}.
\end{proposition}

As a special case of our results, we recover some classical observations
due to Artin (\cite[Thm 1.4]{ArtinLeftIdeals}, cf. the more recent \cite{BocklandtFlatLocus}).
As an example, consider the residue map for a standard conic bundle
$Q \to S$ (cf. \cite{SarkisovConicB} for a definition).
Let $\alpha \in \Br(k(S))[2]$ be the Brauer class of the bundle. 
Let $D$ be an irreducible divisor of $S$ such that $\alpha$ ramifies along $D$.
Since the conic bundle is standard, we can apply Proposition \ref{prop.brauerseveri}
to the local ring $R$ at the generic point of $D$.
Using the notation of the proposition, we have that
$\bar{\alpha}$ is the class of the smooth conic bundle $P$ over all of the root 
stack, which extends the given smooth bundle over the generic point of $\Spec(R)$.
Hence, by comparing it to the singular fiber of the bundle over the generic point of $D$,
we conclude that the residue of $\alpha$ along $D$ can
be seen geometrically as the class of the space of components over its generic point.

An analogous observation is true for standard
Brauer-Severi surface bundles (\cite{MaedaStandard}).

\bibliographystyle{hplain}
\bibliography{brres}

\end{document}